\documentclass[oneside,english]{amsart}
\usepackage[T1]{fontenc}
\usepackage[latin9]{inputenc}
\usepackage{color}
\usepackage{amstext}
\usepackage{amsthm}
\usepackage{amssymb}
\usepackage{esint}

\makeatletter
\numberwithin{equation}{section}
\numberwithin{figure}{section}
\theoremstyle{plain}
\newtheorem{thm}{\protect\theoremname}
\theoremstyle{definition}
\newtheorem{defn}[thm]{\protect\definitionname}
\theoremstyle{plain}
\newtheorem{cor}[thm]{\protect\corollaryname}
\theoremstyle{plain}
\newtheorem{lem}[thm]{\protect\lemmaname}
\theoremstyle{plain}
\newtheorem{prop}[thm]{\protect\propositionname}

\makeatother

\usepackage{babel}
\providecommand{\corollaryname}{Corollary}
\providecommand{\definitionname}{Definition}
\providecommand{\lemmaname}{Lemma}
\providecommand{\propositionname}{Proposition}
\providecommand{\theoremname}{Theorem}

\begin{document}
\title{$\sigma_{2}$ Yamabe problem on conic spheres \mbox{II}:boundary compactness
of the moduli}
\author{Hao Fang}
\address{14 MacLean Hall, Department of Mathematics, University of Iowa, Iowa
City, IA, 52242}
\email{hao-fang@uiowa.edu}
\thanks{H.F.'s work is partially supported by a Simons Foundation Collaboration
Grant in Mathematics}
\author{Wei Wei}
\address{Shanghai Center for Mathematical Science, Fudan University, Shanghai,
China}
\email{wei\_wei@fudan.edu.cn}
\thanks{W.W's work is partially supported by BoXin programme.}
\begin{abstract}
We prove a convergence theorem on the moduli space of constant $\sigma_{2}$
metrics for conic 4-spheres. We show that when a numerical condition
is convergent to the boundary case, the geometry of conic 4-spheres
converges to the boundary case while preserving capacity.
\end{abstract}

\maketitle

\section{introduction}

In this paper, we discuss the moduli space of constant $\sigma_{2}$
metric on a conic 4-sphere, following our previous work \cite{FW}. 

We start with notations and some historical comments. Let $(M,g)$
be a smooth manifold and $Ric,R$ be Ricci curvature and scalar curvature
respectively. The Schouten tensor is defined as 
\[
A_{g}=\frac{1}{n-2}(Ric-\frac{1}{2(n-1)}Rg).
\]
Denote $\{\lambda(A_{g})\}$ as the eigenvalue of $A_{g}$ respect
to $g$. Define $k$-th elementary symmetric function $\sigma_{k}(\lambda)$
as the following 
\[
\sigma_{k}(\lambda)=\sum_{1\le i_{1}<\cdots<i_{k}\le n}\lambda_{i_{1}}\lambda_{i_{2}}\cdots\lambda_{i_{k}}.
\]
When $k=1,$ up to a constant, $\sigma_{1}(\lambda(A_{g}))$ is the
scalar curvature. Define the conformal class 
\[
[g]=\{g_{u}=e^{2u}g|u\in C^{\infty}(M)\}.
\]
The classical Yamabe problem is to find a metric in $[g]$ with constant
scalar curvature. $\sigma_{k}$ Yamabe problem is raised by Viaclovsky
in \cite{V1} to find a metric in $[g]$ with constant $\sigma_{k}$
curvature, which is equivalent to find a solution $u$ to the following
equation

\[
\sigma_{k}(\lambda(A_{g_{u}}))=c,
\]
where $A_{g_{u}}=A_{g}-\nabla^{2}u+\nabla u\otimes\nabla u-\frac{|\nabla u|_{g}^{2}}{2}g$
and $c$ is a constant.

In order to make the equation elliptic, a point-wise positive cone
condition is usually considered as below 
\[
\Gamma_{k}^{+}:=\{\lambda=(\lambda_{1},\cdots,\lambda_{n}),\ s.t.\ \sigma_{1}(\lambda)>0,\cdots,\,\sigma_{k}(\lambda)>0\}.
\]

In the last 20 years, $\sigma_{2}$ Yamabe problems in 4-manifolds
have been studied intensively due to its variational structure and
geometric connection with Gauss-Bonnet-Chern formula. In \cite{CGY1},
Chang-Gursky-Yang have constructed a positive $\sigma_{2}$ curvature
metric by exploring the connection between $\sigma_{2}$ curvature
and $Q$ curvature. Furthermore, they have proved that when the manifold
is not conformal to sphere, there exists a smooth metric such that
its $\sigma_{2}$ curvature admits any given positive function \cite{CGY2}.
In \cite{GS}, using a metric structure on the spaces of conformal
factors, Gursky-Streets have proved that the metric with constant
$\sigma_{2}$ curvature is unique in the same setting. There are abundant
results about the existence and regularity of $\sigma_{k}$ Yamabe
problem and we refer to \cite{GLW,GV2,GV3,GV4,GV5,GW1,GW2,H,TW,Tru}
and references therein. 

Singular $\sigma_{k}$ Yamabe problem considers solutions of same
partial differential equations on a smooth manifold away from a lower
dimensional singular set. For $k=1$, the singular Yamabe problem
has been extensively studied from view points of semi-linear elliptic
equations and differential geometry. For $k=1$, $n=2$, or singular
Yamabe problem on surfaces, see the seminal works of Troyanov \cite{Tr,Tr2},
see also works of Chen-Li \cite{CL,CL1}, Luo-Tian \cite{LT} and
Fang-Lai\cite{FL1,FL2,FL3}. For $k=1$, $n\ge3$, concerning the
positive scalar curvature, in a seminal paper \cite{S}, Schoen has
given the first construction of solutions with finite point singularities.
His proof is further simplified by Mazzeo and Pacard \cite{MF}. Many
works have appeared later. Especially, Mazzeo-Pollack-Uhlenbeck \cite{MPU}
have introduced the moduli space $\mathcal{M}$ consisting of singular
solutions for constant scalar curvature and proved that $\mathcal{M}$
is locally a real analytic variety of formal dimension $k$. For $4\le2k<n$,
Mazzieri-Segatti \cite{LSe} have constructed complete locally conformally
flat manifold with constant positive $k$-curvature. For $n\ge5$,
the existence of metric admitting singularities with constant $\sigma_{2}$
curvature on $n$ dimensional manifolds has been proved by Santos
\cite{SS} with some additional Weyl conditions. Note that solutions
for singular $\sigma_{k}$ Yamabe problem are complete metric for
$k<\frac{n}{2}$. 

We focus on the case $k=\frac{n}{2}.$ In \cite{CHY}, Chang-Han-Yang
have studied the radial solution of positive constant $\sigma_{k}$
curvature equation and noticed that when $k=n/2$, the metric near
singularity is cone like. Furthermore, Han-Li-Teixeira in \cite{HLT}
have proved that in a punctured ball, the solution is close to a radial
solution up to a H$\ddot{o}$lder pertubation, which extends the well
known result of Caffarelli-Gidas-Spruck \cite{CGS} for the $k=1$
case. See also \cite{KMPS}. Thus, solutions of $\sigma_{2}$ Yamabe
problem on a 4-sphere can only be conic. Motivated by these works,
we have introduced the conic version to $\sigma_{2}$ Yamabe problem
in \cite{FW}.

For simplicity, we discuss the standard sphere case here. With the
standard stereographic projection, we define $g=e^{2u}g_{E}$ with
singularities $p_{1},\cdots p_{q-1}\in\mathbb{R}^{4},\ p_{q}=\infty$
and 
\begin{itemize}
\item $u(x)=\beta_{i}\ln|x-p_{i}|+v_{i}(x)$ as $x\rightarrow p_{i}$ for
$i=1,\cdots,\,q-1;$ 
\item $u(x)=(-2-\beta_{q})\ln|x|+v_{\infty}(x)$ as $|x|\rightarrow\infty$, 
\end{itemize}
where $v_{i}(x)$ and $v_{\infty}(x)$ are bounded in their respective
neighborhoods. We call above $g$ conic metric and use $(S^{4},D=\sum_{i=1}^{q}\beta_{i}p_{i},g,g_{0})$
to include all the geometric information, where $g_{0}$ is the standard
sphere metric. By a slight abuse of notation, we also use $D$ to
denote the singular set $\{p_{1},\cdots,p_{q}\}$. We also use $[g_{D}]$
to denote the conformal class with above conic singularities $D$.
The corresponding $\sigma_{2}$ Yamabe equation in $\mathbb{R}^{n}$
is 

\begin{equation}
\sigma_{2}(g_{E}^{-1}A_{ij})=\frac{3}{2}e^{4u},\label{eq:sigma2}
\end{equation}
where 

\[
A_{ij}=-u_{ij}+u_{i}u_{j}-\frac{|\nabla u|^{2}}{2}\delta_{ij}.
\]

In our previous work \cite{FW}, we have studied the conic metric
with constant $\sigma_{2}$ curvature on 4-sphere by constructing
a local mass $M(t)$ along level sets of the conformal factor. See
Section 2 for more details regarding $M(t)${\footnotesize{}.} Similar
to Mazzeo-Pollack-Uhlenbeck \cite{MPU}, we use $\mathcal{M}^{2}(\mathbb{S}^{4})$
to denote a moduli space as the set of all smooth solutions $u$ to
singular $\sigma_{2}$ Yamabe problem $\sigma_{2}(\lambda(A_{g}))=\frac{3}{2}$
on $\mathbb{S}^{4}\backslash P,$ where $P=\{p_{1},p_{2},\cdots,p_{q}\}$
are any fixed singularities.

In \cite{FW}, we give the following definition:
\begin{defn}
\label{def:New}Let $(S^{4},D=\sum_{i=1}^{q}\beta_{i}q_{i},g_{0})$
be a conic 4-sphere with the standard background metric $g_{0}$.
For all $j\in\{1,\cdots,q\}$, we denote $\widetilde{\beta_{j}}:=\bigg(\sum_{1\le i\neq j\le q}\beta_{i}^{3}\bigg)^{1/3}$. 
\end{defn}

\begin{itemize}
\item We call $(S^{4},D)$ subcritical for the $\sigma_{2}$ Yamabe equation
if for any $j=1,\cdots,q$\\
 $\frac{3}{8}\beta_{j}^{2}(\beta_{j}+2)^{2}<\frac{3}{8}\widetilde{\beta_{j}}^{2}(\widetilde{\beta_{j}}+2)^{2}+(\widetilde{\beta_{j}}+\frac{3}{2})(\sum_{1\le i\neq j\le q}\beta_{i}^{2}-\widetilde{\beta_{j}}^{2}),$ 
\item We call $(S^{4},D)$ critical for the $\sigma_{2}$ Yamabe equation
if there exists a $j\in\{1,\cdots,q\}$ such that\\
 $\frac{3}{8}\beta_{j}^{2}(\beta_{j}+2)^{2}=\frac{3}{8}\widetilde{\beta_{j}}^{2}(\widetilde{\beta_{j}}+2)^{2}+(\widetilde{\beta_{j}}+\frac{3}{2})(\sum_{1\le i\neq j\le q}\beta_{i}^{2}-\widetilde{\beta_{j}}^{2})$, 
\item Otherwise, we call $(S^{4},D)$ supercritical for the $\sigma_{2}$
Yamabe equation, which means that there exists a $j\in\{1,\cdots,\,q\},$\\
 $\frac{3}{8}\beta_{j}^{2}(\beta_{j}+2)^{2}>\frac{3}{8}\widetilde{\beta_{j}}^{2}(\widetilde{\beta_{j}}+2)^{2}+(\widetilde{\beta_{j}}+\frac{3}{2})(\sum_{1\le i\neq j\le q}\beta_{i}^{2}-\widetilde{\beta_{j}}^{2}).$
\end{itemize}
In \cite{FW}, we have proved the following theorem:
\begin{thm}
\label{thm:Main theorem-1} Let $(S^{4},\,D,\ g_{0})$ be defined
as above. Assume that $\lambda(A_{g})\in\Gamma_{2}^{+}$. If $(S^{4},D)$
is supercritical, then there does not exist a conformal metric $g\in[g_{D}]$
with positive constant $\sigma_{2}$ curvature. If $(S^{4},D)$ is
critical with positive constant $\sigma_{2}$ curvature, then $(S^{4},g)$
is a football. 
\end{thm}

Here we call $(S^{4},g)$ football if $g=e^{2u}g_{E}$ satisfies (\ref{eq:sigma2})
with two singular points, which is first described in case (a) of
Theorem 1\cite{CHY}. 

Theorem 2 gives some partial results which are parallel to the 2-dimensional
case. But compared to 2-dimensional case, the 4-dimensional one is
fully nonlinear and more complicated. The existence and uniqueness
under the subcritical condition is still unknown.

In this paper, we prove a compactness theorem for the moduli space
$\mathcal{M}^{2}(\mathbb{S}^{4})$, which can be defined as the collection
of all conic metric on $S^{4}$ with constant $\sigma_{2}$ curvature.
This is a parallel result of \cite{FL1} in dimension 2. 

First, we introduce a concept of capacity for the $\sigma_{2}$ Yamabe
problem. Using the stereographic projection from $q\in S^{4},$ any
metric in $\mathcal{M}^{2}(\mathbb{S}^{4})$ can be written as $g_{u}=e^{2u}g_{E}.$
We define 
\begin{align}
K(S^{4},g,q) & =\max_{t\in\mathbb{R}}\frac{1}{|S^{3}|}e^{4t}|\{u\geq t\}|.\label{eq:add52}\\
K(S^{4},g)= & \max_{q\in S^{4}}K(S^{4},g,q).\nonumber 
\end{align}
Note that that for $t\to\pm\infty,$ the quantity in the righthand
side of (\ref{eq:add52}) converges to 0 by the Gauss-Bonnet-Chern
theorem proved in \cite{FW}; see also Theorem \ref{thm: past theorem}.
Thus, $K(S^{4},g,q)$ is well defined. We also point out that $K(S^{4},g,p)$
is scaling and translating invariant. 

Second, we introduce the following 
\begin{defn}
A sequence of subcritical conformal divisors $\{D_{l}=\sum_{i=1}^{q}\beta_{l,i}p_{l,i}\}_{l=1}^{\infty}$
is called numerically convergent and non-degenerate if 
\begin{equation}
\lim_{l\rightarrow\infty}\beta_{l,i}=\beta_{i},\ \ \sum_{i=1}^{q}|\beta_{i}|>0,\label{eq:add13}
\end{equation}
and there exists a universal constant $C>0$ such that 
\begin{equation}
2-\sum_{i=1}^{q}\frac{\beta_{i}^{3}+3\beta_{i}^{2}}{2}\geq C.\label{eq:add12}
\end{equation}
\end{defn}

We remark that if $g_{l}\in(S^{4},D_{l},g_{0})$ has constant $\sigma_{2}$
curvature $\frac{3}{2}$, by the Gauss-Bonnet-Chern theorem proved
in \cite{FW}, (\ref{eq:add12}) implies that there is a constant
$C'>0$ independent of $l$ such that 
\begin{equation}
Vol(S^{4},g_{l})\geq C'.\label{eq:add14}
\end{equation}

Finally, we give the main result of this paper:
\begin{thm}
\label{thm:GHconvergence}Let $\{g_{l}\}_{l=1}^{\infty}\in\mathcal{M}^{2}(\mathbb{S}^{4})$
be a sequence of subcritical metrics with respect to conformal divisors
$D_{l}=\sum_{i=1}^{q}\beta_{l,i}p_{l,i}$. Suppose that $\{D_{l}\}$
is numerically convergent and non-degenerate. If there exists a $j\in\{1,\cdots,q\}$
and some $\epsilon>0$, such that $\widetilde{\beta_{l,j}}:=\bigg(\sum_{1\le i\neq j\le q}\beta_{i}^{3}\bigg)^{1/3}>-1+\epsilon$
and 
\begin{equation}
\lim_{l\to\infty}\frac{3}{8}\beta_{l,j}^{2}(\beta_{l,j}+2)^{2}-[\frac{3}{8}\widetilde{\beta_{l,j}}^{2}(\widetilde{\beta_{l,j}}+2)^{2}+(\widetilde{\beta_{l,j}}+\frac{3}{2})(\sum_{1\le i\neq j\le q}\beta_{l,i}^{2}-\widetilde{\beta_{l,j}}^{2})]=0,\label{eq:add6}
\end{equation}
then there exists a subsequence of $\{g_{l}\}$ converging to $g_{\infty}$
in Gromov-Hausdorff sense, where $g_{\infty}$ is either the standard
4-sphere metric or a football metric defined by Chang-Han-Yang. Furthermore,
$g_{\infty}$ is characterized by the following capacity identity:
\[
K(S^{4},g_{\infty})=\lim_{l\to\infty}K(S^{4},g_{l}).
\]
 
\end{thm}

Similar to the main result of \cite{FL1}, Theorem \ref{thm:GHconvergence}
indicates that in dimension 4, the regular boundary component of $\mathcal{M}^{2}(\mathbb{S}^{4})$
consists of critical metrics of 0 or 2 singular points, which is uniquely
determined by the capacity. It describes a geometric merging of subcritical
conformal divisors when taking the limit. In addition, we have
\begin{cor}
\label{cor:volume}Conditions as in Theorem \ref{thm:GHconvergence}.
We have
\begin{equation}
Vol(S^{4},g_{\infty})\leq\liminf_{l\to\infty}Vol(S^{4},g_{l}).\label{eq:new}
\end{equation}
In particular, if $g_{\infty}$ is the standard 4-sphere metric, we
have $\text{\ensuremath{\beta_{l,i}\to0} }$ for all $i\in\{1,\cdots,q\}$
and $l\to\infty.$ 
\end{cor}

Note that (\ref{eq:new}) in general is not the identity, which implies
that the Gromov-Hausdorff convergence leads to possible volume collapsing
near singular points. This is a new feature that does not happen in
the dimension 2 case.

For $k<n/2$, in \cite{Wei}, the second-named author also gives a
compactness theorem for complete manifolds with constant $\sigma_{k}$-curvature
under some natural conditions. 

Even though the statement of our result is comparable to that of \cite{FL1},
and both proofs rely heavily on the smallness of isoperimetric defects,
there are some key differences between dimension 2 and 4. First, when
the isoperimetric defect of a region is small, the classical Bennesen
inequality in dimension 2 indicates the Gromov-Hausdorff closeness
of the region to a round 2-disc. There is no similar result for higher
dimensions without convexity assumptions. Instead, we apply result
of \cite{FMP} where a measure theoretical closeness is achieved under
similar conditions. Second, due to the non-linear nature, we use the
capacity concept to fix the conformal gauge, which is different from
methods used in \cite{FL1}. These two new key components, together
with standard a priori estimates, lead us to the needed convergence.
However, the non-linear nature of our setup prevents us from getting
more precise information about the convergence, which should be more
closely examined.

We remark that our new method can be used to give an alternative proof
of main results of \cite{FL1}, because of the simpler relation between
volume and capacity in dimension 2.

The authors intend to study the uniqueness of solutions following
approach of Gursky-Streets \cite{GS}. We would like to address the
existence problem eventually. We would also explore the $\sigma_{2}$
Yamabe problem for general conic 4-manifolds.

We organize this paper as below. In Section 2, we follow the setup
of \cite{FW} to discuss level sets of conformal factors. In Section
3, we list some technical results that are used in our work. In Section
4, we prove the main theorem. 

Both authors would like thank a referee who has pointed out mistakes
in a previous version.

\section{Analysis on the level set}

In this section, we recall some definition and results of \cite{FW}
for our discussion.

For a conic metric $(\mathbb{S}^{4},e^{2u}g_{E},D=\sum_{i=1}^{q}\beta_{i}p_{i})$,
assume that $\sigma_{2}(\lambda(g_{u}))=\frac{3}{2}$ on $\mathbb{S}^{4}\backslash D.$
We begin with the following definition about the level sets of the
conformal factor $u:$
\begin{align*}
L(t) & =\{x:u=t\}\subset\mathbb{\mathbb{R}}^{4},\\
S(t) & =\{x:u\geq t\}\subset\mathbb{R}^{4}.
\end{align*}

For a fixed $t,$ it is clear that $L(t)$ is smooth at a point $P\in L(t)$
if and only if $\nabla u\neq0.$ Define 
\[
S=\{x\in\mathbb{R}^{4}\backslash\{p_{1},\cdots,p_{q}\}|\,\nabla u(x)=0\}.
\]

In \cite{FW}, we have proved the following lemma, which implies the
properness of the integral on the level set.
\begin{lem}
\label{lem:measure 0}\cite{FW}For $U\subset\mathbb{R}^{4},$ if
$\sigma_{2}(g_{u})\neq0$ in $U$, then $\mathcal{H}^{3}(S\cap U)=\mathcal{H}^{4}(S\cap U)=0,$
where \textup{$\mathcal{H}^{3}$,} $\mathcal{H}^{4}$ is the 3-dimensional
and 4-dimensional Hausdorff measure. 
\end{lem}

We adopt the local coordinates in \cite{FW}. Denote 

\begin{equation}
A(t)=\fint_{S(t)}e^{4u},\label{eq:add31}
\end{equation}
\begin{equation}
B(t)=\fint_{S(t)}dx=\frac{|S(t)|}{|S^{3}|},\label{eq:add32}
\end{equation}
\begin{equation}
C(t)=e^{4t}B(t),\label{eq:add33}
\end{equation}
and 
\begin{equation}
z(t)=-(\fint_{L(t)}|\nabla u|^{3}\ dl)^{\frac{1}{3}},\label{eq:add34}
\end{equation}
and finally,

\[
D(t)=\frac{1}{4}(\fint_{L(t)\backslash S}\{2H|\nabla u|^{2}-2|\nabla u|^{3}\}\ dl).
\]
Here $H$ is the mean curvature of the level set $L(t),$ and $dl$
is the induced 3-dimensional measure on $L(t).$ In this paper, $\fint$
means $\frac{1}{|S_{3}|}\int$ and $\omega_{1}=|B_{1}|=\frac{1}{4}|S_{3}|.$ 

To see how singularities affect the shape of a manifold, we introduce
a mass $M(t)$ and some related results, which are proved in our previous
paper \cite{FW}.

Define
\begin{equation}
M(t)=\frac{2}{3}D(t)+\frac{4}{9}D(t)z(t)+\frac{1}{36}z^{4}(t)-C(t).\label{eq:add62}
\end{equation}

\begin{thm}
\label{thm: past theorem}\cite{FW}Let $u$ be the solution to (\ref{eq:sigma2})
with $(S^{4},g_{u},D)$, we have 

\begin{equation}
\int_{\mathbb{R}^{4}}\sigma_{2}(g_{u}^{-1}A_{u})dx=(2-\sum_{i=1}^{q}\frac{\beta_{i}^{3}+3\beta_{i}^{2}}{2})|S^{3}|,\label{eq:gauss-bonnet}
\end{equation}

\begin{equation}
M'(t)\ge0,\label{eq:monotone}
\end{equation}
 
\begin{equation}
E=M'+4C,\label{eq: new form}
\end{equation}
 
\begin{equation}
A(t)=\frac{2}{3}[D(t)-D(+\infty)],\label{eq:add15}
\end{equation}

\begin{equation}
A'(t)=-\fint_{L(t)}\frac{e^{4t}}{|\nabla u|}dl,\label{eq:A'}
\end{equation}

\begin{equation}
\frac{1}{3}\frac{d}{dt}(z^{3})=\fint_{L(t)}(\frac{H}{3}|\nabla u|-\nabla_{nn}u)|\nabla u|,\label{eq:key equality}
\end{equation}
and 

\begin{equation}
(A')^{2}\fint_{L(t)}(H|\nabla u|-\frac{3}{2}|\nabla u|^{2})|\nabla u|\fint_{L(t)}(\frac{H}{3}|\nabla u|-\nabla_{nn}u)|\nabla u|dl\ge\frac{3}{2}e^{12t}|L(t)|^{4}\frac{1}{|S_{3}|^{4}},\label{eq:key inequality}
\end{equation}
where 
\[
E=\frac{1}{3}(2zA'+\frac{2}{3}z'\fint_{L(t)}(H|\nabla u|-\frac{3}{2}|\nabla u|^{2})|\nabla u|).
\]
\end{thm}

(\ref{eq:gauss-bonnet}) is Theorem 3 in \cite{FW}; the proof of
(\ref{eq:monotone})(\ref{eq: new form}) is given in Theorem 14 \cite{FW};
(\ref{eq:key equality}) is from Lemma 12 in \cite{FW};(\ref{eq:key inequality})
is contained in the proof of Lemma 14 in \cite{FW}; and (\ref{eq:add15})
is proved in Theorem 15 of \cite{FW}. Here $z$ is also increasing
respect to $t$.

Now as $A$ is strictly decreasing with respect to $t$ by Theorem
\ref{thm: past theorem}, we may use $A$ as a variable for functions
$z(A),C(A),M(A),D(A)$, respectively. See \cite{FL2,FL3} for some
general discussion. It is clear that $A\in[0,V]$, where $V=\frac{2}{3}(2-\sum_{i=1}^{q}\frac{\beta_{i}^{3}+3\beta_{i}^{2}}{2})$.
Under this setting, we have the following simple estimate:
\begin{lem}
\label{lemC_A} $\frac{\partial C}{\partial A}\ge z+1$ for $A\in(0,V)$
a.e..
\end{lem}

\begin{proof}
By (\ref{eq:A'}) and (\ref{eq:add33}),
\[
(A'z)^{3}=(e^{4t}\fint_{L(t)}|\nabla u|^{-1})^{3}(\fint_{L(t)}|\nabla u|^{3}\ dl)\ge e^{12t}|L(t)|^{4}\geq[4C(t)]{}^{3},
\]
where the last inequality is due to the iso-perimetric inequality.
We then obtain that 
\[
A'z\ge4C(t).
\]
 Together with the simple computation $C'=4C+A'$, we get 
\[
A'(z+1)\ge C'(t).
\]
Noting that $A'<0,$ we have thus get our result.
\end{proof}

\section{Technical Preparation}

In this section we first list some results that will be used in our
proof. 

The classical Bonnesen inequality for 2-dimensional case indicates
that isoperimetric deficit of a bounded set $E$ can control the Hausdorff
distance between $E$ and a ball, but it is difficult to get the corresponding
one for higher dimension. The correct setting is measure theoretical.
In \cite{FMP}, Fusco, Maggi and Pratelli have proved the following
sharp theorem, which is previous studied by Fuglede \cite{Fuglede}
and Hall \cite{Hall}.
\begin{thm}
\label{thm:quantaities about isoperimetric equalities} For any $E\subset\mathbb{R}^{n}$with
a finite measure, there exists a constant $\gamma(n)>0$ such that 

\begin{equation}
\alpha(E)^{2}\leq\gamma(n)\frac{|\partial E|-|\partial B_{r}|}{|\partial B_{r}|}\label{eq:add7}
\end{equation}
where $\alpha(E):=\min_{x\in\mathbb{R}^{n}}\left\{ \frac{\left|E\Delta B_{r}(x)\right|}{\left|B_{r}\right|}:|E|=\left|B_{r}\right|\right\} $
, $E\Delta B_{r}(x)=(E\backslash B_{r}(x))\cup(B_{r}(x)\backslash E)$
and $|\partial E|$ denotes $n-1$ dimensional measure of boundary.
\end{thm}

Next, we list some estimates for $\sigma_{k}$ Yamabe equations, which
will be used later. 
\begin{thm}
\label{lem:local esytimates}Let $u\in C^{4}$ be the solution to
$\sigma_{k}(g_{u}^{-1}A_{g_{u}})=C_{0}$ in a geodesic ball $B_{r}^{n}$,
where $C_{0}$ is a positive constant. We have 

\[
\sup_{B_{r/2}}\left(|\nabla u(x)|^{2}+\left|\nabla^{2}u(x)\right|\right)\leq C(1/r^{2}+\sup_{x\in B_{r}}e^{2u})
\]
where $C$ depends only on $n,k$.
\end{thm}

The local estimates can be found in \cite{LL1} and have been proved
by various mathematicians \cite{GLW,GW1,LL1,SC,W1,STW}. 

We now present a simple uniform volume estimate.
\begin{prop}
\label{cor:area control}Let $u$ be the solution to (\ref{eq:sigma2})
with conic singularities $\sum_{i=1}^{q}\beta_{i}p_{i}$,where $-1<\beta_{i}\leq0$
and $q\in\mathbb{Z}^{+}$. For every $t\in\mathbb{R}$, the level
set $S(t)=\{x:u(x)>t\}$ has the following 
\begin{equation}
|S(t)|\le\frac{4}{3}e^{-4t}|S^{3}|.\label{eq:add11}
\end{equation}
\end{prop}

\begin{proof}
We have 
\begin{equation}
\int_{\mathbb{R}^{4}}e^{4u}\ge\int_{\{u>t\}}e^{4u}\ge e^{4t}|S(t)|.\label{eq:add9}
\end{equation}
By (\ref{eq:gauss-bonnet}), we have
\begin{equation}
\int_{\mathbb{\mathbb{R}}^{4}}\sigma_{2}(g_{u}^{-1}A_{g_{u}})dx=|S^{3}|(2-\sum_{i=1}^{q}\frac{\beta_{i}^{3}+3\beta_{i}^{2}}{2})\leq2|S^{3}|.\label{eq:add10}
\end{equation}
Now combining (\ref{eq:sigma2}), (\ref{eq:add9}) and (\ref{eq:add10}),
we have proved (\ref{eq:add11}).
\end{proof}
Next we prove a local upper estimates for singular solution for $n/2$
Yamabe equation in $\mathbb{R}^{n}$, which will be used later. The
following behavior near singularity is well known in the classical
Yamabe problem and we refer to Schoen and Pollack \cite{Pollack}.
\begin{thm}
\label{thm:Uniform decay}Let $n$ be even and $g=e^{2u}g_{E}$ such
that $\sigma_{n/2}(g^{-1}A_{g})=C_{0}>0$, on $B(0,1)\backslash\{0\}$,
and $\lambda(A_{g})\in\Gamma_{n/2}^{+}$. If 
\begin{equation}
\int_{B_{1}}e^{nu}dx<\chi(\mathbb{S}^{n})|S^{n-1}|=2|S^{n-1}|,\label{eq:88}
\end{equation}
then for $x\in B(0,\frac{1}{2})\backslash\{0\}$, 
\[
u(x)\le C-\ln|x|,
\]
where $C$ is independent of $u$.
\end{thm}

\begin{proof}
For $x_{0}\in B(0,1/2)\backslash\{0\}$, let $\rho(x)=d(x,x_{0})$
and $\sigma=d(x_{0,}0)/2$, where $d$ is the Euclidean distance,
we define 
\begin{equation}
h(x,x_{0})=\ln(\sigma-\rho(x))+u(x).\label{eq:89}
\end{equation}
Since $u$ is smooth up to $\partial B(x_{0},\sigma)$, and $h(x)=-\infty$
for $x\in\partial B(x_{0},\sigma)$, the maximum of $h$ happens in
$B(x_{0},\sigma)$. 

We claim $h(x,x_{0})<C$, where $C$ is independent of $x_{0}$ and
$u$. We prove this claim by contradiction. Otherwise, we have a sequence
$x_{0,l}\in B(0,1/2)\backslash\{0\},x_{l}\in B(x_{0,l},\sigma_{l}),\ g_{l}=e^{2v_{l}}g_{E},\int_{B_{1}}e^{nv_{l}}dx<\chi(\mathbb{S}^{n})|S^{n-1}|,$
$\sigma_{l}=d(x_{0,l},0)/2$, $\rho_{l}(x)=d(x,x_{0,l})$ such that
for $\forall l$, 

\[
\ln(\sigma_{l}-\rho_{l}(x_{l}))+v_{l}(x_{l})=h_{l}(x_{l})=\max_{B(x_{0,l},\sigma_{l})}h_{l}>\frac{2}{n-2}\ln l.
\]

Since $(\sigma_{l}-d(x_{l},x_{0,l}))\le\sigma_{l}\le\frac{1}{2},$
we have $v_{l}(x_{l})\rightarrow\infty$ as $l\rightarrow\infty$.
Denoting $\lambda_{l}=e^{v_{l}(x_{l})},$ we get that $\lambda_{l}\rightarrow\infty$
as $l\rightarrow\infty.$ Let 
\[
\tilde{v}_{l}(z)=-\ln\lambda_{l}+v_{l}(x_{l}+\frac{z}{\lambda_{l}}),
\]
with $\tilde{v}_{l}(0)=0$ and $g_{l}=e^{2\tilde{v}_{l}(z)}g_{E}$.

We have $v_{l}(x)\le\ln2+v_{l}(x_{l})$ in $B(x_{l},r_{l}/2)\subset B_{\sigma_{l}}(x_{0,l})$,
where $r_{l}=\sigma_{l}-\rho_{l}(x_{l})$. So for $R_{l}=\lambda_{l}\frac{r_{l}}{2}=\frac{1}{2}(e^{\frac{n-2}{2}v_{l}(x_{l})}(\sigma_{l}-\rho_{l}(x_{l}))^{\frac{n-2}{2}})^{\frac{2}{n-2}}\ge\frac{1}{2}l^{\frac{2}{n-2}}\to\infty,$
\[
\tilde{v}_{l}(z)\le\ln2\quad in\quad|z|\le R_{l}.
\]
 Thus $\tilde{v}_{l}(z)$ are uniformly bounded in any compact set
in $\mathbb{R}^{n}$ and $\tilde{v}_{l}(0)=0$. By Theorem \ref{lem:local esytimates},
for any compact $K\subset\mathbb{R}^{n}$, $\tilde{v}_{l}(x)\ge C$
for $l$ large and $x\in K$. By the classical Schauder theory, we
get 
\[
\sup_{K}|\tilde{v}_{l}|_{C^{2,\alpha}}\le C.
\]

Then by Arzela-Ascoli theorem, there exists a subsequence of $\widetilde{v}_{l}$
(still denoted by $\widetilde{v}_{l}$) and a function $v_{\infty}$
such that $\tilde{v}_{l}\rightarrow v_{\infty}$ in $C^{2,\alpha}$
for any compact sets in $\mathbb{R}^{n}$. Therefore 
\[
\sigma_{k}(\lambda(A_{e^{2v_{\infty}}g_{E}}))=C_{0}\quad{\rm in}\quad\mathbb{R}^{n}.
\]
By the Liouville theorem of Li-Li\cite{LL2}, $g_{\infty}=e^{2v_{\infty}}g_{E}$
is the standard sphere metric.

For $k=n/2$, we consider the Gauss-Bonnet-Chern formula of \cite{FW},
and see that

\begin{align*}
 & \lim_{l\rightarrow\infty}\int_{B(0,1)}\sigma_{k}(\lambda(A_{g_{v_{l}}}))dvol_{g_{v_{l}}}\ge\lim_{l\rightarrow\infty}\int_{B(x_{l},r_{l}/2)}\sigma_{k}(\lambda(A_{g_{v_{l}}}))dvol_{g_{v_{l}}}\\
 & \ge\lim_{l\rightarrow\infty}\int_{B(0,R_{l})}\sigma_{k}(\lambda(A_{g_{\tilde{v}_{l}}}))dvol_{g_{\tilde{v}_{l}}}\\
 & =\int_{\mathbb{R}^{n}}\sigma_{k}(\lambda(A_{g_{v_{\infty}}}))dvol_{g_{v_{\infty}}}=\chi(\mathbb{S}^{n})|S^{n-1}|,
\end{align*}
which contradicts with (\ref{eq:88}).

Now we have proved that (\ref{eq:88}) for $|x_{0}|<\frac{1}{2}$.
Picking $x=x_{0},$ we get $u(x_{0})\le C-\ln|x_{0}|$ for $|x_{0}|<\frac{1}{2}$.
We have finished the proof.
\end{proof}
For a similar result regarding the general $\sigma_{k}$ Yamabe equation,
see \cite{Wei}. 

\section{Proof of main theorem }

In this section, we prove our main theorem. 

Note that we have a sequence of conformal metrics $\{g_{l}\}_{l=1}^{\infty}$
with constant $\sigma_{2}$ curvature and prescribed non-degenerate
conic singularities that satisfies (\ref{eq:add6}). Without loss
of generality, we assume $j=q$. The key point of the proof is to
fix the proper conformal gauge so that resulting conformal factors
on $\mathbb{R}^{4}$ have enough information to carry out a priori
estimates.

First, by proper Kelvin transforms, we may assume that $p_{q}=\infty$
and write $g_{l}=e^{2u_{l}}g_{E}$ with respect to the Euclidean metric
of $\mathbb{R}^{4}$. In order to have sub-convergence of $u_{l}$,
we will determine a coordinate system for each $l$ later, which is
crucial for our proof. We use notations introduced in Section 2 and
use subscript $l$ to denote corresponding geometric terms for $g_{l}.$
Also, throughout our proof, we pass to a subsequence of $g_{l}$ whenever
necessary and we will denote the subsequence as still $g_{l}$ to
simplify notations.

\begin{thm}
\label{thm: key1}Conditions as in Theorem \ref{thm:Main theorem-1},
then for $t\in\mathbb{R}$ almost everywhere, after passing to a subsequence, 

\begin{equation}
\lim_{l\rightarrow\infty}|S_{l}(t)|^{3}\alpha^{2}(S_{l}(t))=0,\label{eq:add1}
\end{equation}
and 
\begin{equation}
\lim_{l\rightarrow\infty}|L_{l}(t)|^{4}\frac{1}{|S_{3}|^{4}}-(4B_{l}(t))^{3}=0,\label{eq:add2}
\end{equation}
where $\alpha$ is defined in Theorem \ref{thm:quantaities about isoperimetric equalities}.
\end{thm}

\begin{proof}
We do the following computation.

\begin{align}
 & (M_{l}')^{3}+12(M_{l}')^{2}C_{l}+48C_{l}^{2}M_{l}'\nonumber \\
 & =M_{l}'(t)\{(M_{l}'+4C_{l})^{2}+4C_{l}(t)[M_{l}'(t)+4C_{l}(t)]+(4C_{l})^{2}\}\nonumber \\
 & =(E_{l}(t)-4C_{l}(t))(E_{l}^{2}+4C_{l}(t)E_{l}+(4C_{l})^{2})\nonumber \\
 & =\big(\frac{1}{3}(2z_{l}A_{l}'+\frac{2}{3}z_{l}'\fint_{L_{l}(t)}\sigma_{1}(\widetilde{A})|\nabla u_{l}|)\big)^{3}-(4C_{l}(t))^{3}\nonumber \\
 & \ge\frac{2}{3}(A_{l}')^{2}\fint_{L_{l}(t)}\sigma_{1}(\widetilde{A})|\nabla u_{l}|\fint_{L_{l}(t)}(\frac{H}{3}|\nabla u_{l}|-\nabla_{nn}u_{l})|\nabla u_{l}|dl-(4C_{l}(t))^{3}\nonumber \\
 & \ge e^{12t}|L_{l}(t)|^{4}\frac{1}{|S_{3}|^{4}}-e^{12t}(4B_{l}(t))^{3}\label{eq:key}\\
 & \ge c(n)e^{12t}|S_{l}(t)|^{3}\alpha^{2}(S_{l}(t))\nonumber 
\end{align}
where 
\[
E_{l}=\frac{1}{3}(2z_{l}A'_{l}+\frac{2}{3}z_{l}'\fint_{L_{l}(t)}\sigma_{1}(\widetilde{A})|\nabla u_{l}|).
\]
In (\ref{eq:key}), the fourth inequality holds due to H$\ddot{o}$lder
inequality and (\ref{eq:key equality}), while the fifth inequality
is due to (\ref{eq:key inequality}) and the last inequality is due
to Theorem \ref{thm:quantaities about isoperimetric equalities}.

Furthermore, 

\begin{align}
 & \int_{t_{0}}^{+\infty}\big((M_{l}'(t))^{3}+12C_{l}(M_{l}')^{2}+48C_{l}^{2}M_{l}')^{1/3}dt\label{eq:inequality 3}\\
\le & \int_{t_{0}}^{\infty}M_{l}^{'}+(12C_{l})^{1/3}(M_{l}')^{2/3}+(48C_{l}^{2}M_{l}')^{1/3}dt\nonumber \\
\le & \int_{t_{0}}^{\infty}M_{l}^{'}dt+(\int_{t_{0}}^{\infty}M_{l}^{'})^{2/3}(\int_{t_{0}}^{+\infty}12C_{l}dt)^{1/3}+(48)^{1/3}(\int_{t_{0}}^{\infty}M_{l}^{'}dt)^{1/3}(\int_{t_{0}}^{+\infty}C_{l}dt)^{2/3}.\nonumber 
\end{align}
As $\int_{-\infty}^{+\infty}4C_{l}(t)dt=\int_{\mathbb{R}^{4}}e^{4u_{l}}<C<\infty$
by Theorem \ref{thm: past theorem}, with (\ref{eq:key}) and (\ref{eq:inequality 3}),
we get

\begin{align*}
\int_{t_{0}}^{\infty}M_{l}^{'}dt+(\int_{t_{0}}^{\infty}M_{l}^{'})^{2/3}(\int_{t_{0}}^{+\infty}12C_{l}dt)^{1/3}+(48)^{1/3}(\int_{t_{0}}^{\infty}M_{l}^{'}dt)^{1/3}(\int_{t_{0}}^{+\infty}C_{l}dt)^{2/3}\\
\ge\int_{t_{0}}^{\infty}\big(c(n)e^{12t}|S_{l}(t)|^{3}\alpha^{2}(S_{l}(t))\big)^{1/3}dt.
\end{align*}
Here $M_{l}$ is increasing respect to $t$ by Theorem \ref{thm: past theorem}.
By (\ref{eq:add6}) and (\ref{eq:monotone}), we have 

\begin{equation}
\lim_{l\rightarrow\infty}[M_{l}(+\infty)-M_{l}(-\infty)]=0,\label{eq:add61}
\end{equation}
where $M_{l}(\pm\infty)=\lim_{t\to\pm\infty}M_{l}(t)$. Thus, for
$M_{l}'(t)\geq0,$

\[
\int_{-\infty}^{+\infty}M_{l}'(t)dt=M_{l}(+\infty)-M_{l}(-\infty)\rightarrow0
\]
 as $l\to\infty.$ Thus, $\int_{t_{0}}^{\infty}M_{l}'(t)\rightarrow0.$
Then $\int_{t_{0}}^{\infty}\big(e^{12t}c(n)|S_{l}(t)|^{3}\alpha^{2}(S_{l}(t))\big)^{1/3}dt$
as $l\rightarrow\infty.$ 

After passing to a subsequence $|S_{l}(t)|^{3}\alpha^{2}(S_{l}(t))\ge0$
converges to $0$ almost everywhere for $t\ge t_{0}.$ Repeating the
same argument for a sequence of $t_{i}\rightarrow-\infty$ and a diagonal
argument, we have proved that $|S_{l}(t)|^{3}\alpha^{2}(S_{l}(t))$
converges to $0$ as $l\to\infty$ almost everywhere on $\mathbb{R}.$
Thanks to Sard's theorem, we have proved the theorem.
\end{proof}
Denote $U\subset\mathbb{R}^{4}$ to be the dense set on which (\ref{eq:add1})
and (\ref{eq:add2}) hold. We are now ready to prove our main theorem.
\begin{proof}
 (of Theorem \ref{thm:GHconvergence}) Note that we have fixed our
infinity as $p_{l,q}.$ We are ready for the second step of gauge
fixing. As the scaling $u(x)\to u(kx)+\ln k$ does not change the
conic metric and the capacity, we choose appropriate scaling for each
$u_{l}$ such that,
\begin{equation}
K(S^{4},g_{l},p_{l,q})=C_{l}(0).\label{eq:add51}
\end{equation}

We claim that 
\begin{equation}
\limsup_{l\to\infty}C_{l}(0)>0.\label{eq:75}
\end{equation}

We prove the claim by contradiction. If not, then $\lim_{l\rightarrow\infty}C_{l}(0)=0$.
By (\ref{eq:add51}) and the definition (\ref{eq:add52}), $C_{l}(t)$
converges uniformly to $C_{\infty}(t)=0$. Consider $M_{l}(A_{l})$
as function of $A_{l},$ where $A_{l}\in[0,V_{l}]$ and $V_{l}=\frac{2}{3}(2-\sum_{i=1}^{q}\frac{\beta_{l,i}^{3}+3\beta_{l,i}^{2}}{2})$.
By Theorem \ref{thm: past theorem}, $M_{l}$ is monotone respect
to $A_{l}$. Applying the key condition (\ref{eq:add6}), we get 
\[
\lim_{l\rightarrow\infty}M_{l}(A)=\frac{1}{4}\beta_{\infty}^{2}(2+\beta_{\infty})^{2}
\]
uniformly for $A\in[0,V_{\infty}].$ Furthermore, the limit function
$z_{\infty}(A)=\lim_{l\rightarrow\infty}z_{l}(A)$ can be characterized
by the definition (\ref{eq:add62}) and satisfies the following 
\begin{equation}
A+\frac{2}{3}D_{\infty}(+\infty)+\frac{2}{3}(A+\frac{2}{3}D_{\infty}(+\infty))z_{\infty}(A)+\frac{1}{36}z_{\infty}^{4}(A)=\frac{1}{4}\beta_{\infty}^{2}(2+\beta_{\infty})^{2}.\label{eq:71}
\end{equation}
where $D_{\infty}(+\infty)=\lim_{l\rightarrow\infty}D_{l}(+\infty)=\frac{3}{2}\sum_{i=1}^{q-1}\beta_{i}^{2}-\frac{1}{2}\sum_{i=1}^{q-1}|\beta_{i}^{3}|$,
and $z_{\infty}(0)=(\sum_{i=1}^{q-1}\beta_{i}^{3})^{1/3}>-1+\epsilon>-1.$
There exists $\delta>0$ such that 
\begin{equation}
z_{\infty}(A)>-1,\label{eq:72}
\end{equation}
for $A\in(0,\delta)$.

By Lemma \ref{lemC_A}, for each $l$, $\frac{dC_{l}}{dA_{l}}\ge z_{l}(A_{l})+1$
a.e., which leads to 
\begin{equation}
\frac{dC_{\infty}(A)}{dA}\geq z_{\infty}(A)+1.\label{eq:73}
\end{equation}
Thus, for $A\in(0,\delta),$ we combine (\ref{eq:72}) and (\ref{eq:73})
to get $C_{\infty}(\delta)>0,$ which is a contradiction to the fact
that $C_{\infty}(A)=0.$ We have now proved (\ref{eq:75}). We thus
write
\begin{equation}
\lim_{l\to\infty}C_{l}(0)=C_{\infty}(0)>0.\label{eq:80}
\end{equation}

Now, fix any $t_{0}\in U$ such that $t_{0}<0$, and consider level
sets $S_{l}(t_{0})$. By Theorem \ref{thm: key1}, $\alpha(S_{l}(t_{0}))\rightarrow0$
as $l\rightarrow\infty$, which implies that there exists a ball $B(x_{l}^{*},r_{l}^{*})$
such that 
\begin{equation}
\lim_{l\to\infty}|\Delta(S_{l}(t_{0}),B(x_{l}^{*},r_{l}^{*}))|=0.\label{eq:77}
\end{equation}
 By (\ref{eq:80}) and Corollary \ref{cor:area control}, we have
the following crucial fact
\begin{equation}
r_{t_{0}}=\lim_{l\to\infty}r_{l}^{*}\in(0,+\infty).\label{eq:78}
\end{equation}

The third and final step of gauge fixing is to apply the translation
$u(x)\to u(x+k)$ for each $u_{l}$ such that $x_{l}^{*}=x_{t_{0}}$
is independent of $l$. (\ref{eq:77}) and (\ref{eq:78}) imply that
as $l\to\infty,$
\begin{equation}
|\Delta(S_{l}(t_{0}),B(x_{t_{0}},r_{t_{0}}))|\to0.\label{eq:102}
\end{equation}

Now for any $t\in U$, by Theorem \ref{thm: key1}, $\alpha(S_{l}(t))\rightarrow0$.
Following arguments in \cite{FL1}, we claim that we find a ball $B(x_{t},r_{t})$
such that as $l\to\infty,$
\begin{equation}
|\Delta(S_{l}(t),B(x_{t},r_{t}))|\to0.\label{eq:add3}
\end{equation}
Here the existence of $r_{t}$ is due to Theorem \ref{thm:quantaities about isoperimetric equalities}
and by (\ref{eq:add3}), $r_{t}<+\infty.$ It is clear that $r_{t}$
is monotone non-increasing with respect to $t$. 

We prove the existence of $x_{t}$. Define $U'=\{t\in U,$$\ r_{t}>0\}$.
For any $t\in U',$ by Theorem \ref{thm: key1}, there exists a sequence
of balls $B(x_{l,t}^{*},r_{l,t}^{*})$ such that
\begin{equation}
\lim_{l\to\infty}|\Delta(S_{l}(t),B(x_{l,t}^{*},r_{l,t}^{*}))|=0.\label{eq:101}
\end{equation}
It is easy to see that $r_{l,t}^{*}\to r_{t}.$ By (\ref{eq:101}),(\ref{eq:102})
and our choice of $x_{t_{0}}$, we get 
\[
\lim_{l\rightarrow\infty}|(B(x_{t_{0}},r_{t_{0}})\cap B(x_{l,t}^{*},r_{l,t}^{*}))|=\frac{|\min\{r_{t_{0}},r_{t}\}|^{4}}{4}|S^{3}|>0,
\]
which, together with crucial facts that $r_{t_{0}}>0$ and $r_{t}>0,$
implies that
\begin{equation}
\limsup_{l\to\infty}|x_{t_{0}}-x_{l,t}^{*}|=\limsup_{l\to\infty}|x_{l}^{*}-x_{l,t}^{*}|\leq\max\{r_{t_{0}},r_{t}\}.\label{eq:201}
\end{equation}
(\ref{eq:102}) clearly implies that when moving to a subsequence,
we can take limit of $x_{l,t}^{*}$, which we pick as $x_{t}.$ If
$U'=U,$ we have thus finished the proof of (\ref{eq:add3}). Otherwise,
let $t'=\sup U'$ and $r'=\inf\{r_{t},\ t\in U'\}$, we see that $\cap_{t\in U'}B(x_{t},r_{t})$
is a ball of radius $r'\geq0.$ Pick any $y\in\cap_{t\in U'}B(x_{t},r_{t})$.
For $t\in U\setminus U'$, we have $\lim_{l\to\infty}|S_{l}(t)|=0,$
which also implies that $r_{t}=0.$ We may simply define $x_{t}=y$.
We have now finished the proof of (\ref{eq:add3}).

By Proposition \ref{cor:area control} and (\ref{eq:add3}), we see
that
\[
\lim_{s\in U,s\to+\infty}r_{s}=0.
\]
It is also clear that $\cap_{t\in U}B(x_{t},r_{t})$ is a single point,
which we denote as $Q$.

Consider the set $S'=\{p_{1},\cdots,p_{q'}\}\setminus\{\infty\}$,
where $p_{i}=\lim_{l\to\infty}p_{l,i}$ exists if we pass to a subsequence.
Note that $S'$ may be empty. Thus, for any compact set $K\subset\mathbb{R}^{4}\setminus S'$,
by (\ref{eq:gauss-bonnet}) and Theorem \ref{thm:Uniform decay},
possibly replaced by a subsequence, $\{u_{l}\}$ has uniform upper
bound on $K$. By Lemma \ref{lem:local esytimates}, there exists
a $u_{\infty}$ such that $u_{l}\rightarrow u_{\infty}$ in $C_{loc}^{1,1}(\mathbb{R}^{4}\setminus S').$
We claim that $u_{\infty}(x)\neq-\infty$ for $x\notin S$. If not,
exists a $x'\notin S'$ such that $u_{\infty}(x')=-\infty$. Since
$\mathbb{R}^{4}\setminus S'$ is connected, by Theorem \ref{lem:local esytimates},
$u_{\infty}=-\infty$ in $\mathbb{R}^{4}\setminus S'$. This contradicts
with the fact that $|S_{l}(t_{0})|=|B(x_{t_{0}},r_{t_{0}})|>0,$ and
$S_{l}(t_{0})$ is convergent in measure theoretical sense.

Therefore, $u_{l}$ has a uniform lower bound on any compact set away
from $S'$. By classical Schauder estimates,
\begin{equation}
u_{l}\rightarrow u_{\infty}\ {\rm in}\ C^{\infty}(\mathbb{R}^{4}\setminus S'),\label{eq:add4}
\end{equation}
 where $u_{\infty}$ satisfies $\sigma_{2}(A_{u_{\infty}})=3/2$.

\textcolor{black}{We now analyze }$S_{\infty}(t)=\{u_{\infty}\geq t\}$
and $L_{\infty}(t)=\{u_{\infty}=t\}$.\textcolor{black}{{} Let $a=\max\{|x|,x\in S\}$.
Fix $t\in U$ and for any compact ball $K=\overline{B(0,r)}\supset B(x_{t},r_{t})$
with $r>a+1$, by (\ref{eq:add3}) and (\ref{eq:add4}), $L_{l}(t)\cap K$
converges to $L_{\infty}(t)\cap K$ in $C^{\infty}$ sense, which
implies that $S_{\infty}(t)\cap K=B(x_{t},r_{t}),$ when $r_{t}>0$.
Let $l\to\infty,$ we get that for $t\in U'$, then $r_{t}>0,$ and
\begin{equation}
S_{\infty}(t)=B(x_{t},r_{t}).\label{eq:300}
\end{equation}
}When $t\in U\setminus U'$, we may conclude that $S_{\infty}(t)$
is either empty or isolated points. This in particular implies 
\begin{equation}
r'=\liminf_{t\in U'}r_{t}=0.\label{eq:900}
\end{equation}
Hence, $\cap_{t\in U}S_{\infty}(t)$ consists of at most one point,
which we denote as $Q$ if it exists. Since $u_{\infty}$ is smooth
away from $S'$, we have proved that $S'$ contains at most one point.
Thus, $u_{\infty}$ is a smooth solution of $\sigma_{2}(A_{u_{\infty}})=3/2$
on $\mathbb{R}^{4}\setminus S'$, which by Theorem \ref{thm:Main theorem-1},
has to be rotationally symmetric. Furthermore, $g_{\infty}=e^{2u_{\infty}}g_{E}$
is either a standard 4-sphere metric or a football metric with conformal
divisor $\beta(Q)+\beta(\infty)$. Another consequence of (\ref{eq:add4})
is 
\[
K(S^{4},g_{\infty})=\lim_{l\to\infty}K(S^{4},g_{l}).
\]
\end{proof}
Finally, we prove Corollary \ref{cor:volume}. 
\begin{proof}
It is clear from our discussion that $\lim_{l\to\infty}C_{l}(t)=C_{\infty}(t),$
which is smooth. Using the fact that $C'_{l}(t)=4C_{l}(t)+A'_{l}(t),$
we get, for any finite $s<t,$ 
\begin{equation}
A_{\infty}(t)-A_{\infty}(s)=\lim_{l\to\infty}A_{l}(t)-A_{l}(s),\label{eq:82}
\end{equation}
which only implies that 
\begin{equation}
Vol(g_{\infty})\leq\liminf_{l\rightarrow\infty}Vol(g_{l}).\label{eq:final}
\end{equation}
This corresponds to possible volume collapsing, which is a new behavior
in dimension 4. 

If $Vol(g_{\infty})=2|S^{3}|$ is the volume of the standard 4-sphere,
we have, by (\ref{eq:gauss-bonnet}) and (\ref{eq:final}), that $\beta_{l,i}\to0$
when $l\to\infty.$
\end{proof}

\end{document}